\documentclass[11pt]{amsart}

\usepackage[margin=1.1in]{geometry}
\usepackage{amsmath,amssymb,comment}
\usepackage[expansion=false]{microtype}
\usepackage{tikz}
\usetikzlibrary{arrows.meta}
\usepackage[colorlinks=true,linkcolor=blue,citecolor=blue,urlcolor=blue]{hyperref}

\newtheorem{theorem}{Theorem}
\newtheorem{proposition}[theorem]{Proposition}
\newtheorem{lemma}[theorem]{Lemma}
\newtheorem{corollary}[theorem]{Corollary}
\theoremstyle{remark}
\newtheorem{remark}[theorem]{Remark}

\newcommand{\R}{\mathbb{R}}
\newcommand{\N}{\mathbb{N}}

\title{The Erd\H{o}s Similarity Conjecture for Two-Fold Sumsets with a Geometric Summand}

\author{N. Mora Cuellar}
\address{Department of Mathematics, University of British Columbia}
\email{natalia.mora@math.ubc.ca}

\author{A. Iosevich}
\address{Department of Mathematics, University of Rochester, Rochester, NY, USA}
\email{iosevich@gmail.com}

\author{N. Kulkarni}
\address{Department of Mathematics, University of Rochester, Rochester, NY, USA}
\email{nkulkar7@math.rochester.edu}

\author{I. Rojas Aravena}
\address{Department of Mathematics, University of British Columbia}
\email{i.andres@math.ubc.ca}

\author{A. Yavicoli}
\address{Department of Mathematics, University of British Columbia}
\email{yavicoli@math.ubc.ca}

\date{}

\begin{document}

\begin{abstract}
We settle a major case in the two-set regime of the Erd\H{o}s similarity
conjecture: the sum of a geometric sequence and an arbitrary infinite set is
never measure universal. Here a set $E\subset\R$ is measure universal if every
measurable set of positive Lebesgue measure contains an affine copy of $E$.
More precisely, if $A\subset\R$ is infinite, $a\neq 0$, and $0<|r|<1$, then
neither
\[
\{ar^n:n\ge 1\}+A
\qquad\text{nor}\qquad
\{ar^n:n\ge 1\}-A
\]
is measure universal. More generally, the same conclusion holds when the
geometric sequence is replaced by any set containing a lacunary sequence
$(b_n)$ with $-\log b_n=O(n)$.
Bourgain proved non-universality for sums of three arbitrary infinite sets,
whereas the two-set regime is one of the principal remaining cases.
Crucially, our conclusion applies to $\{2^{-n}\}+A$ for every infinite
$A$, even though the non-universality of $\{2^{-n}\}$ itself remains open.

The arbitrary-summand theorem is the maximal lacunary-density endpoint of a
general counting-function trade-off. If $S_1,S_2\subset\R$ contain lacunary
subsequences and $I(W),J(W)$ count their terms that are at least $e^{-W}$, then
$S_1+S_2$ and $S_1-S_2$ are not measure universal whenever
\[
\limsup_{W\to\infty}\frac{I(W)J(W)}{W}=\infty.
\]
No scale-separation or relative-decay assumption is required. The proof combines
a finite-grid implementation of Kolountzakis' criterion with a
near-additive-energy estimate controlling the clustering of lacunary cross-sums.
A packing-number variant replaces lacunarity on one factor by a quantitative
metric-mass condition. In particular, for $\alpha_1,\alpha_2>0$, the
stretched-exponential sumset
\[
\{2^{-n^{\alpha_1}}\}+\{2^{-n^{\alpha_2}}\}
\]
is not measure universal whenever
$1/\alpha_1+1/\alpha_2>1$; analogous conclusions hold for difference sets.
\end{abstract}

\maketitle

\section{Introduction}

Let $E\subset\R$. We say that $E$ is \emph{measure universal} if for every
measurable $L\subset\R$ with $m(L)>0$ there exist $x\in\R$ and $\lambda\neq 0$
with
\[
x+\lambda E\subset L .
\]
Determining which infinite sets are measure universal is the Erd\H{o}s similarity
problem. By the Lebesgue density theorem every finite set is measure universal,
and Erd\H{o}s conjectured that these are the only ones: no infinite set is measure
universal. The conjecture remains open. Since no unbounded set can be measure
universal, it is enough to consider bounded infinite sets, each of which has an
accumulation point. After an affine change of
variables and passage to an infinite subsequence, every such set contains a
strictly decreasing sequence $a_n\to 0$. By monotonicity of measure universality
under inclusion and affine invariance, the conjecture therefore reduces to showing
that every strictly decreasing null sequence of positive numbers is not measure
universal.

In this reduced setting the decisive parameter is the rate of decay. Falconer
\cite{Falconer} and, independently, Eigen \cite{Eigen} proved that sublacunary
sequences --- those with $a_{n+1}/a_n\to 1$ --- are not measure universal.
Kolountzakis \cite{Kol97} extended this to any sequence containing a subset with
suitably large gaps (Theorem~\ref{thm:kol} below). The most difficult unresolved
regime includes genuinely lacunary decay. The prototype is a single geometric
sequence $(\lambda^k)$ with $\lambda\in(0,1)$; even $\{2^{-n}\}$ is not known to
be non-universal.

A second source of examples comes from additive structure. Bourgain \cite{Bourgain}
proved that for any infinite $A_1,A_2,A_3\subset\R$ the sumset $A_1+A_2+A_3$ is
not measure universal. This draws a sharp line between sums of three sets and sums
of two: the two-set regime is one of the main open cases. For two sets,
Kolountzakis' finite-gap criterion provides the principal general source of
examples relevant to the present paper. In particular, as recorded in
the survey \cite{JLM}, the criterion \cite{Kol97} gives that
\begin{equation}\label{eq:stretched}
\{2^{-n^{\alpha}}\}+\{2^{-n^{\alpha}}\}\ \text{is not measure universal for every }
\alpha\in(0,2);
\end{equation}
the case $\alpha=1$ is the geometric double sumset
$\{2^{-n}\}+\{2^{-n}\}$. To see other results related to this problem, see the
surveys \cite{Svetic} and \cite{JLM}.

\subsection{Results}

Our principal result settles a major case in the two-set regime of the
Erd\H{o}s similarity conjecture: the sum or difference of a geometric sequence
and an arbitrary infinite set is not measure universal. More generally, the
same conclusion holds when the geometric sequence is replaced by a set
containing a lacunary sequence $(b_i)$ with $-\log b_i=O(i)$. We obtain this
as the endpoint of a general counting-function theorem for two lacunary
summands, proved through a finite-grid implementation of Kolountzakis'
criterion and a near-additive-energy estimate.

Our organizing tool is a finite-block mechanism for proving non-universality. The
central task becomes the construction of large finite blocks at bounded scale with
controlled separation. Proposition~\ref{prop:block} shows that any set containing
arbitrarily large such blocks is not measure universal. The gap tolerance
$e^{-o(n)}$ is exactly the strength already present in \cite{Kol97}; we isolate it
in a form suited to the finite grids constructed
below.

We call a sequence $(b_i)_{i\ge 1}\subset(0,\infty)$ \emph{lacunary with ratio
$q\in(0,1)$} if $b_{i+1}\le q\,b_i$ for all $i$; such a sequence is strictly
decreasing and tends to $0$. Its \emph{counting function} is
\[
I(W):=\#\{\,i:\ b_i\ge e^{-W}\,\},\qquad W>0 .
\]
All our hypotheses are phrased through counting functions; pointwise decay
conditions such as $-\log b_i=o(i^2)$ enter only as convenient sufficient
conditions for lower bounds on them. Our first main result is the following.

\begin{theorem}\label{thm:product}
Let $S_1,S_2\subset\R$ contain lacunary subsequences $(b_i)_{i\ge1}$ and
$(d_j)_{j\ge1}$ in $(0,\infty)$, with counting functions $I$ and $J$. If
\begin{equation}
\limsup_{W\to\infty}\frac{I(W)\,J(W)}{W}=\infty,
\tag{$\star$}\label{eq:star}
\end{equation}
then $S_1+S_2$ is not measure universal. The same conclusion holds for $S_1-S_2$.
\end{theorem}

Two features deserve emphasis. First, the two lacunary subsequences may live on
completely unrelated scales: no scale-separation or relative-decay hypothesis
relates them, and their counting functions may trade off against each other.
Second, \eqref{eq:star} is a $\limsup$ condition: all decay hypotheses in this
paper need only hold along a sequence of scales, because the underlying criterion
of Kolountzakis requires blocks only for arbitrarily large $n$
(Theorem~\ref{thm:kol}). The main new ingredient is a near-additive-energy
estimate (Lemma~\ref{lem:energy}) showing that the cross-sums of two lacunary
sequences have uniformly bounded clustering at every scale; as a consequence, a
positive proportion of the cross-sums can be selected so as to be mutually
separated. This avoids the need for any scale-separation hypothesis between
the two sequences.

The counting function of \emph{any} lacunary sequence satisfies $I(W)=O(W)$
(Lemma~\ref{lem:counting}(i)), so a single lacunary summand can never contribute
more than $\asymp W$ points above scale $e^{-W}$, and equality of growth,
$I(W)\asymp W$, is equivalent to the pointwise condition $-\log b_i=O(i)$
(Remark~\ref{rem:dyadic}). At this endpoint, condition \eqref{eq:star} degenerates
to $J(W)\to\infty$, which holds for every infinite null sequence whatsoever. This
observation yields our second main result, in which the second summand is a
completely arbitrary infinite set.

\begin{theorem}\label{thm:arbitrary}
Let $S\subset\R$ contain a lacunary subsequence $(b_i)_{i\ge1}\subset(0,\infty)$
with $-\log b_i=O(i)$. Then $S+A$ and $S-A$ are not measure universal for every
infinite set $A\subset\R$. In particular, $\{2^{-n}\}+A$ is not measure universal
for any infinite $A\subset\R$.
\end{theorem}

The significance is that this conclusion does not follow by monotonicity from
non-universality of the structured summand: it applies to
$S=\{2^{-n}\}$, whose own non-universality remains open. It therefore
occupies an intermediate position between the standard two-sequence results
discussed above, which impose conditions on both summands, and Bourgain's
three-set theorem, which imposes none. It arises at the linear-counting
endpoint of Theorem~\ref{thm:product}: once $I(W)\gtrsim W$, the condition
on the second counting function reduces to $J(W)\to\infty$.

In fact, lacunarity itself is needed on only one of the two factors; the other
needs only \emph{metric mass}, measured by packing numbers. For a set
$S\subset\R$ and a bounded interval $Q_0$ with $S\cap Q_0$ infinite, let
\[
N(W):=\max\big\{\#U:\ U\subseteq S\cap Q_0,\
|u-u'|\ge 2e^{-W}\ \text{for all distinct } u,u'\in U\big\}
\]
denote the packing function of $S$ in $Q_0$; it is finite for each $W$,
nondecreasing, and tends to infinity.

\begin{theorem}\label{thm:packing}
Let $S_2\subset\R$ contain a lacunary sequence $(d_j)_{j\ge1}\subset(0,\infty)$,
of any ratio $q\in(0,1)$, with counting function $J$. Let $S_1\subset\R$ be
infinite, and suppose some bounded interval $Q_0$ contains infinitely many points
of $S_1$ (otherwise $S_1$ is unbounded and $S_1\pm S_2$ is trivially not
universal); let $N$ be the packing function of $S_1$ in $Q_0$. If
\begin{equation}\label{eq:packstar}
\limsup_{W\to\infty}\frac{\min\big(N(W),\,J(W)\big)\,J(W)}{W}=\infty,
\end{equation}
then $S_1+S_2$ and $S_1-S_2$ are not measure universal.
\end{theorem}

The packing form contains Theorem~\ref{thm:arbitrary}. Indeed, when
$J(W)\gtrsim W$, condition \eqref{eq:packstar} reduces to
$\min(N(W),J(W))\to\infty$, which is automatic, so the packing form yields the
arbitrary-summand theorem uniformly. It complements Theorem~\ref{thm:product}
rather than subsuming it: the first factor is freed from all structural
hypotheses --- it may consist of clusters with no usable lacunary subsequence, so
that condition \eqref{eq:star} is unavailable --- but its contribution to the
count is capped at $J(W)$ (see Remark~\ref{rem:capNJ} for an explanation of
this cap).
The same near-energy estimate that drives Theorem~\ref{thm:product} thus reaches
clustered first factors directly, through the packing function alone and with no
relative-density input; Section~\ref{sec:packing} illustrates this.

Theorems~\ref{thm:product}--\ref{thm:packing} recover the standard
two-sequence examples recorded in the literature and yield several asymmetric
families. We record the main consequences; throughout, ``not universal''
abbreviates ``not measure universal''.

\begin{corollary}\label{cor:tradeoff}
Let $S_1,S_2\subset\R$ contain lacunary subsequences $(b_i)$, $(d_j)$ in
$(0,\infty)$ such that either
\begin{enumerate}
\item[(i)] $-\log b_i=o(i^{\alpha})$ and $-\log d_j=o(j^{\beta})$ for some
$\alpha,\beta>1$ with $\dfrac1\alpha+\dfrac1\beta\ge 1$; or
\item[(ii)] $-\log b_i=O(i^{\alpha})$ and $-\log d_j=O(j^{\beta})$ for some
$\alpha,\beta\ge 1$ with $\dfrac1\alpha+\dfrac1\beta>1$.
\end{enumerate}
Then $S_1+S_2$ and $S_1-S_2$ are not universal.
\end{corollary}

The case $\alpha=\beta=2$ of Corollary~\ref{cor:tradeoff}(i) is the symmetric
quadratic-decay theorem, which was the main result of an earlier version of this
paper; we record it separately together with its specialization to double sumsets.

\begin{corollary}\label{cor:symmetric}
Let $S_1,S_2\subset\R$ contain lacunary subsequences $(b_i)$, $(d_j)$ with
$-\log b_i=o(i^2)$ and $-\log d_j=o(j^2)$. Then $S_1+S_2$ and $S_1-S_2$ are not
universal. In particular, if $S$ contains a lacunary subsequence $(b_k)$ with
$-\log b_k=o(k^2)$, then $S+S$ and $S-S$ are not universal; every geometric double
sumset is not universal, and \eqref{eq:stretched} holds.
\end{corollary}

\begin{corollary}\label{cor:concrete}
Let $A\subset\R$ be an arbitrary infinite set.
\begin{enumerate}
\item[(a)] For any $a\neq 0$ and $0<|r|<1$, the sets $\{ar^{n}\}+A$ and
$\{ar^{n}\}-A$ are not universal.
\item[(b)] If $(a_n)$ is a decreasing null sequence with
$\liminf_n a_{n+1}/a_n>0$ (which in particular implies
$a_n\ge e^{-Cn}$ for some $C>0$ and all large $n$), then
$\{a_n\}+A$ and $\{a_n\}-A$ are not universal. In particular,
$\{a_n\}+\{a_n\}$ and $\{a_n\}-\{a_n\}$ are not universal.
\end{enumerate}
\end{corollary}

\begin{corollary}\label{cor:stretched}
For $\alpha_1,\alpha_2>0$, the sumset
$\{2^{-n^{\alpha_1}}\}+\{2^{-n^{\alpha_2}}\}$ is not universal whenever
\[
\frac{1}{\alpha_1}+\frac{1}{\alpha_2}>1 .
\]
This contains the symmetric range $\alpha_1=\alpha_2\in(0,2)$ and, more
generally, every pair $\alpha_1,\alpha_2\in(0,2)$. It also includes genuinely
asymmetric pairs outside that square; for instance $\alpha_1=6/5$,
$\alpha_2=5$ is admissible. If $\alpha_1\le 1$, every infinite
second summand is admissible, by Corollary~\ref{cor:concrete}(b).
\end{corollary}

\begin{remark}\label{rem:cantor}
Our results concern sums of two \emph{sequences}, whose truncations at scale
$e^{-W}$ contain only $I(W)J(W)$ points. Sets with \emph{exponentially} many
points per scale are far easier, and for them much stronger results are known.
Every digit set
\[
C(a):=\Big\{\,\sum_{n\ge1}\epsilon_n a_n:\ \epsilon_n\in\{0,1\}\,\Big\}
\qquad(a_n>0\ \text{summable}),
\]
and in particular every symmetric Cantor set, is a sum of \emph{three} infinite
sets, obtained by grouping the digits according to the residue of $n$ modulo
$3$ (each of the three residue classes being infinite); hence $C(a)$ is not
universal by Bourgain's theorem \cite{Bourgain}, with
no decay condition whatsoever. (This observation is recorded in
\cite{KolCantor}, where it is attributed to the referee.) For uncountable sets,
positive Newhouse thickness \cite{GLW} or positive Hausdorff dimension (see
\cite[\S3]{JLM}) also rules out universality. More recently,
Shmerkin and Yavicoli \cite{SY} proved full-measure non-universality under
positive Hausdorff or packing logarithmic-dimension hypotheses, covering many
sets of zero Hausdorff dimension. Alternatively,
applying Theorem~\ref{thm:kol} to the full $n$-th generation of a symmetric
Cantor set with gap lengths $d_n>\ell_n$ --- $2^n$ points with controlled gaps ---
yields non-universality
whenever the generation lengths satisfy $-\log\ell_{n_k}=o(2^{n_k})$ along a
subsequence \cite[Theorem~1.2]{KolCantor}, and a probabilistic refinement
produces avoiding sets of \emph{full} measure when
$-\log\ell_n=o\big(2^{n^{1-\epsilon}}\big)$ \cite[Theorem~1.3]{KolCantor}. By
contrast, even when the digit sequence $(a_n)$ is lacunary, applying
Theorem~\ref{thm:product} only to the two-digit grid
$\{a_{2i}+a_{2j-1}\}_{i,j\ge1}$ would require $-\log a_n=o(n^2)$, a far more
restrictive condition on the digit scales; we therefore make no claims about
Cantor sets. The comparison highlights the regime addressed in this paper: sets with
\emph{polynomially} many points per scale, where neither Bourgain's additive
mechanism nor exponential generation counts are available.
\end{remark}

Corollaries~\ref{cor:tradeoff}--\ref{cor:stretched} are derived in
Section~\ref{sec:corollaries}. No scale-separation condition and no comparison
between the decay rates of the two sequences is required.

Figure~\ref{fig:roadmap} summarizes the logical dependencies among the results of
this paper and the order in which they are proved.

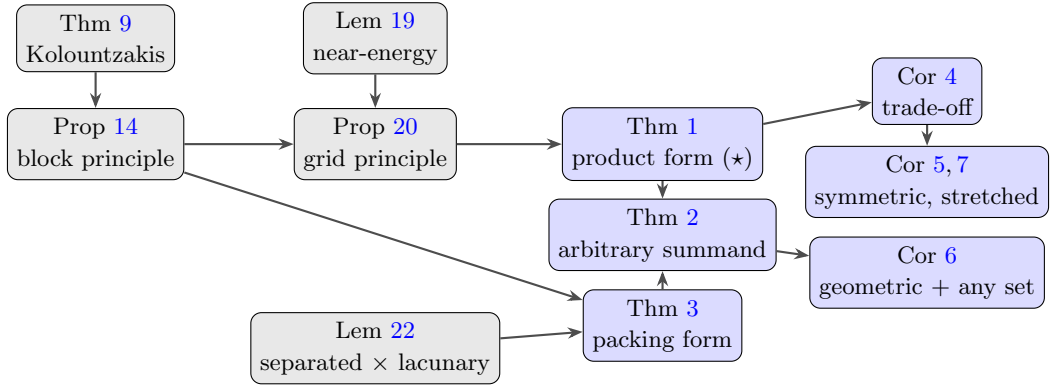
\begin{figure}[t]
\centering
\begin{tikzpicture}[
  >={Stealth[length=2mm]},
  mach/.style={draw, rounded corners, align=center, font=\footnotesize,
               inner sep=3.5pt, fill=black!9},
  res/.style={draw, rounded corners, align=center, font=\footnotesize,
              inner sep=3.5pt, fill=blue!14},
  arr/.style={->, thick, black!70},
]
\node[mach] (kol)   at (0,5.0) {Thm~\ref{thm:kol}\\Kolountzakis};
\node[mach] (block) at (0,3.6) {Prop~\ref{prop:block}\\block principle};
\node[mach] (energy) at (3.7,5.0) {Lem~\ref{lem:energy}\\near-energy};
\node[mach] (grid)   at (3.7,3.6) {Prop~\ref{prop:grid}\\grid principle};
\node[mach] (penergy) at (3.7,0.9) {Lem~\ref{lem:packenergy}\\separated $\times$ lacunary};
\node[res] (thm1) at (7.5,3.6) {Thm~\ref{thm:product}\\product form $(\star)$};
\node[res] (thm2) at (7.5,2.4) {Thm~\ref{thm:arbitrary}\\arbitrary summand};
\node[res] (thm3) at (7.5,1.2) {Thm~\ref{thm:packing}\\packing form};
\node[res] (cor4)  at (11.0,4.3) {Cor~\ref{cor:tradeoff}\\trade-off};
\node[res] (cor57) at (11.0,3.1) {Cor~\ref{cor:symmetric},\,\ref{cor:stretched}\\symmetric, stretched};
\node[res] (cor6)  at (11.0,1.9) {Cor~\ref{cor:concrete}\\geometric $+$ any set};
\draw[arr] (kol)--(block);
\draw[arr] (energy)--(grid);
\draw[arr] (block)--(grid);
\draw[arr] (grid)--(thm1);
\draw[arr] (block)--(thm3);
\draw[arr] (penergy)--(thm3);
\draw[arr] (thm1)--(thm2);
\draw[arr] (thm3)--(thm2);
\draw[arr] (thm1)--(cor4);
\draw[arr] (cor4)--(cor57);
\draw[arr] (thm2)--(cor6);
\end{tikzpicture}
\caption{Roadmap of the paper. Grey boxes are the underlying machinery
(Section~\ref{sec:block} and Sections~\ref{sec:counting}--\ref{sec:packing});
shaded boxes are the non-universality results. Arrows denote logical implication.
Kolountzakis' criterion (Theorem~\ref{thm:kol}) gives the block principle
(Proposition~\ref{prop:block}), which together with the near-energy estimate
drives the grid principle (Proposition~\ref{prop:grid}) and hence
Theorem~\ref{thm:product}; the packing energy (Lemma~\ref{lem:packenergy}) drives
Theorem~\ref{thm:packing}. Both Theorem~\ref{thm:product} (at its endpoint
$\alpha=1$) and Theorem~\ref{thm:packing} (which contains it) yield
Theorem~\ref{thm:arbitrary}.}
\label{fig:roadmap}
\end{figure}

Throughout, $m(\cdot)$ denotes Lebesgue measure on $\R$, and $\log$ is the natural
logarithm. We write $A\subseteq S_1+S_2$ for the sumsets of extracted subsequences
used in the proofs; the conclusions transfer to the full sumsets by monotonicity.

\section{The block principle}\label{sec:block}

We begin with the criterion of Kolountzakis on which everything rests. We use the
form stated in \cite[Theorem~1.1]{KolCantor}; in particular, the chains of points
are required to exist only for \emph{arbitrarily large} $n$, not for every $n$.

\begin{theorem}[Kolountzakis]
\label{thm:kol}
Let $A\subset\R$ be infinite. Suppose that there are integers $n_k\to\infty$
and, for each $k$, points
$a_{k,1}>a_{k,2}>\dots>a_{k,n_k}>0$ in $A$ such that, with
\[
\delta_k:=\min_{1\le i<n_k}
\frac{a_{k,i}-a_{k,i+1}}{a_{k,1}-a_{k,n_k}},
\]
one has $-\log\delta_k=o(n_k)$. Then $A$ is not measure universal.
\end{theorem}

\begin{remark}\label{rem:kolform}
The subsequence formulation is essential here: the criterion requires the chains
only for arbitrarily large cardinalities, as is also used in
\cite[Theorem~1.2]{KolCantor}. This is what allows the hypotheses in the present
paper to be stated as $\limsup$ conditions.
\end{remark}

Three elementary observations will be used repeatedly. The first records that
measure universality is monotone under inclusion; it is what lets us pass from a
convenient sub-sumset to the full sumset.

\begin{lemma}\label{lem:mono}
If $A\subseteq B\subset\R$ and $A$ is not measure universal, then $B$ is not
measure universal. Consequently, if $A\subseteq S_1+S_2$ is not measure universal
then neither is $S_1+S_2$.
\end{lemma}

\begin{proof}
If $x+\lambda B\subset L$ then $x+\lambda A\subset x+\lambda B\subset L$; thus
``$B$ measure universal'' implies ``$A$ measure universal'', and the
contrapositive is the assertion.
\end{proof}

The second records that measure universality is an affine invariant; it lets us
translate and reflect the sets under consideration at will.

\begin{lemma}\label{lem:affine}
Let $\rho\neq 0$ and $\tau\in\R$. Then $E\subset\R$ is measure universal if and
only if $\rho E+\tau$ is measure universal.
\end{lemma}

\begin{proof}
For any $x\in\R$ and $\lambda\neq 0$,
\[
x+\lambda(\rho E+\tau)=(x+\lambda\tau)+(\lambda\rho)E,
\]
and as $(x,\lambda)$ ranges over $\R\times(\R\setminus\{0\})$, so does
$(x+\lambda\tau,\lambda\rho)$. Hence $E$ and $\rho E+\tau$ have exactly the same
families of affine copies.
\end{proof}

The third is the standard complementary (duality) form, included for context.

\begin{lemma}\label{lem:dual}
Let $E\subset\R$ and let $L\subset\R$ be measurable. Then $x+\lambda E\not\subset
L$ for every $x\in\R$ and every $\lambda\neq 0$ if and only if $\lambda E+L^c=\R$
for every $\lambda\neq 0$.
\end{lemma}

\begin{proof}
Assume the first statement and fix $\lambda\neq 0$. If $y\notin\lambda E+L^c$,
then for every $t\in E$ we have $y-\lambda t\notin L^c$, i.e.\ $y-\lambda t\in L$;
thus $y-\lambda E\subset L$, a contradiction. Conversely, assume the second and
fix $x,\lambda$. If $x+\lambda E\subset L$, then $x\notin-\lambda E+L^c$,
contradicting the hypothesis with $-\lambda$.
\end{proof}

We now state the principle that drives all of our results. It is a convenient
repackaging of Theorem~\ref{thm:kol}: rather than listing the points one by one,
we produce finite blocks of arbitrarily large cardinality inside a bounded
interval. The point is the gap tolerance: a block of $n$ points may have minimal
gap as small as $e^{-o(n)}$, and blocks are needed only along a sequence of
cardinalities.

\begin{proposition}\label{prop:block}
Let $A\subset\R$ be bounded and infinite. Suppose that there are integers
$n_k\to\infty$, real numbers $\psi_k=o(n_k)$, and finite sets $\Phi_k\subset A$
such that
\[
|\Phi_k|\ge n_k,
\qquad
\min_{\substack{x,y\in\Phi_k\\x\neq y}}|x-y|\ge e^{-\psi_k}.
\]
Then $A$ is not measure universal.
\end{proposition}

\begin{proof}
Choose $\tau\in\R$ and $M>0$ such that $A+\tau\subset(0,M]$. Translation does
not change the cardinalities or gaps of the blocks, so it suffices, by
Lemma~\ref{lem:affine}, to prove that $A+\tau$ is not measure universal. Replacing
$A$ and $\Phi_k$ by their translates, assume that $A\subset(0,M]$.

For each $k$, choose $n_k$ points from $\Phi_k$ and list them in decreasing order,
$a_1>\cdots>a_{n_k}>0$. Then
\[
a_i-a_{i+1}\ge e^{-\psi_k}
\quad (1\le i<n_k),
\qquad
a_1-a_{n_k}\le M.
\]
Consequently,
\[
\delta_{n_k}
=\min_{1\le i<n_k}\frac{a_i-a_{i+1}}{a_1-a_{n_k}}
\ge \frac{e^{-\psi_k}}{M},
\]
and therefore
\[
-\log\delta_{n_k}\le \psi_k+\log M=o(n_k).
\]
Theorem~\ref{thm:kol} now applies.
\end{proof}

\begin{remark}[Why the $o(n)$ condition matters]\label{rem:tolerance}
The distinction between $o(n)$ and $O(n)$ already contains a central open case of
the conjecture. Indeed, the single geometric sequence produces blocks at the
$O(n)$ level: taking $a_i=2^{-i}$ for $1\le i\le n$ gives $a_1=\tfrac12$,
minimal gap $2^{-n}$, and hence
$-\log\delta_n=\log(2^{n-1}-1)=O(n)$. Thus a version
of Theorem~\ref{thm:kol} with $O(n)$ in place of $o(n)$ would imply that
$\{2^{-n}\}$ is not measure universal.
\end{remark}

\section{Grids of two lacunary sequences}\label{sec:grids}

Throughout this section, $(b_i)_{i\ge1}$ and $(d_j)_{j\ge1}$ denote lacunary
sequences in $(0,\infty)$ with ratios $q_1,q_2\in(0,1)$ respectively, and $I$, $J$
denote their counting functions,
\[
I(W):=\#\{i:\ b_i\ge e^{-W}\},\qquad J(W):=\#\{j:\ d_j\ge e^{-W}\} .
\]
Since lacunary sequences are strictly decreasing,
$\{i:\ b_i\ge e^{-W}\}=\{1,\dots,I(W)\}$ is an initial segment, and $I$ is
nondecreasing in $W$.

\subsection{Counting functions and reduction}\label{sec:counting}

\begin{lemma}\label{lem:counting}
Let $(b_i)$ be lacunary with ratio $q\in(0,1)$ and counting function $I$.
\begin{enumerate}
\item[(i)] $I(W)\le\dfrac{W}{\log(1/q)}+C_0$ for all $W>0$, with
$C_0:=1+\dfrac{\max(0,\log b_1)}{\log(1/q)}$. In particular $I(W)=O(W)$.
\item[(ii)] If $-\log b_i\le C i^{\alpha}$ for some $C>0$, $\alpha\ge 1$ and all
$i\ge i_0$, then $I(W)\ge (W/C)^{1/\alpha}-i_0$ for all $W>0$; for $\alpha=1$ this
gives $I(W)\ge W/(2C)$ for all sufficiently large $W$.
\item[(iii)] If $-\log b_i=o(i^{\alpha})$ for some $\alpha>0$, then
$I(W)/W^{1/\alpha}\to\infty$ as $W\to\infty$.
\end{enumerate}
\end{lemma}

\begin{proof}
(i) If $I(W)=0$, the assertion is immediate. Otherwise write $I=I(W)$.
Lacunarity gives $b_I\le q^{I-1}b_1$, while $b_I\ge e^{-W}$; hence
$(I-1)\log(1/q)\le W+\log b_1\le W+\max(0,\log b_1)$, which is the
claim.

(ii) Every integer $i\in[i_0,(W/C)^{1/\alpha}]$ satisfies
$-\log b_i\le Ci^{\alpha}\le W$, i.e.\ $b_i\ge e^{-W}$; the number of such $i$ is
at least $(W/C)^{1/\alpha}-i_0$.

(iii) Fix $\delta>0$; there is $i_0$ with $-\log b_i<\delta i^{\alpha}$ for
$i\ge i_0$. By the computation in (ii), $I(W)\ge(W/\delta)^{1/\alpha}-i_0$, hence
$\liminf_{W\to\infty}I(W)/W^{1/\alpha}\ge\delta^{-1/\alpha}$. As $\delta>0$ was
arbitrary, $I(W)/W^{1/\alpha}\to\infty$.
\end{proof}

\begin{remark}\label{rem:dyadic}
For a lacunary sequence, $-\log b_i=O(i)$ is equivalent to
$\liminf_{W\to\infty}I(W)/W>0$. Indeed, one direction is
Lemma~\ref{lem:counting}(ii) with $\alpha=1$; conversely, if $I(W)\ge cW$ for all
large $W$, then taking $W=i/c$ gives $b_i\ge e^{-i/c}$ for all large $i$. By
Lemma~\ref{lem:counting}(i) this is the fastest possible growth of a lacunary
counting function, so the hypothesis of Theorem~\ref{thm:arbitrary} asks that
$(b_i)$ be as dense as a lacunary sequence can be. Equivalently, in terms of the
set $S$ alone: $S$ satisfies the hypothesis of Theorem~\ref{thm:arbitrary} if and
only if $S$ meets the dyadic block $[2^{-m-1},2^{-m})$ for a set of integers $m$
of positive lower density. (Given such blocks, pick one point per block and pass
to every other selected block to obtain a lacunary subsequence with ratio
$\le\tfrac12$ and counting function $\gtrsim W$. Conversely, a lacunary
subsequence with counting function $\gtrsim W$ has only $O(1)$ terms in each
dyadic block and therefore occupies a positive proportion of the blocks up to
level $W$.)
\end{remark}

\begin{lemma}[Thinning]\label{lem:thinning}
Let $(b_i)$ be lacunary with ratio $q$ and counting function $I$, and let
$\ell\in\N$. Then $(b_{\ell i})_{i\ge1}$ is lacunary with ratio $q^{\ell}$ and its
counting function is $I_{\ell}(W)=\lfloor I(W)/\ell\rfloor$. Consequently, if a
pair of lacunary sequences satisfies \eqref{eq:star}, then so does any pair of
thinned subsequences $(b_{\ell i})$, $(d_{\ell' j})$.
\end{lemma}

\begin{proof}
Lacunarity of the thinned sequence is clear. Since the index set
$\{i:b_i\ge e^{-W}\}$ is the initial segment $\{1,\dots,I(W)\}$, we have
$b_{\ell i}\ge e^{-W}$ if and only if $\ell i\le I(W)$, whence
$I_{\ell}(W)=\lfloor I(W)/\ell\rfloor$. For the last claim, first note that
\eqref{eq:star} forces $I(W)\to\infty$ and $J(W)\to\infty$: if, say, $I$ were
bounded, then by Lemma~\ref{lem:counting}(i) applied to $(d_j)$ we would have
$I(W)J(W)=O(W)$, contradicting \eqref{eq:star}. Hence for all large $W$,
$\lfloor I(W)/\ell\rfloor\ge I(W)/(2\ell)$ and
$\lfloor J(W)/\ell'\rfloor\ge J(W)/(2\ell')$, so
\[
\limsup_{W\to\infty}\frac{I_{\ell}(W)J_{\ell'}(W)}{W}\ \ge\
\limsup_{W\to\infty}\frac{I(W)J(W)}{4\ell\ell'\,W}\ =\ \infty .
\]
\end{proof}

By Lemma~\ref{lem:thinning} (with $\ell$
chosen so that $q_1^{\ell}<\tfrac12$, and similarly for $(d_j)$) we may and do
assume from now on that
\[
q:=\max(q_1,q_2)<\tfrac12 ,
\]
at the cost of replacing the sequences by thinned subsequences; condition
\eqref{eq:star} is preserved, sumsets of the thinned subsequences are contained in
those of the original ones, and pointwise decay hypotheses of the form
$-\log b_i\le Ci^{\alpha}$ are preserved with $C$ replaced by $C\ell^{\alpha}$.
This reduction is harmless only because two sequences are summed: thinning a
single sequence cannot help, since a lone lacunary sequence has only $I(W)=O(W)$
separated points above scale $e^{-W}$ whatever its ratio --- the obstruction
isolated in Remark~\ref{rem:tolerance}.

\subsection{The near-energy lemma}\label{sec:energy}

The heart of the matter is that the cross-sums have uniformly controlled
clustering. Define, for $W>0$, the near-additive-energy
\[
E(W):=\#\Big\{(i,i',j,j'):\ i,i'\le I(W),\ j,j'\le J(W),\
\big|(b_i-b_{i'})+(d_j-d_{j'})\big|\le e^{-W}\Big\}
\]
and the signed near-energy
\[
E^-(W):=\#\Big\{(i,i',j,j'):\ i,i'\le I(W),\ j,j'\le J(W),\
\big|(b_i-b_{i'})-(d_j-d_{j'})\big|\le e^{-W}\Big\}.
\]
\begin{lemma}\label{lem:energy}
With $q=\max(q_1,q_2)<\tfrac12$ there is a constant $C=C(q)$ such that
\[
E(W) = E^-(W) \le C\,I(W)\,J(W)\qquad\text{for all }W>0 .
\]
\end{lemma}

We use one structural fact repeatedly: for a lacunary sequence of ratio $\le q$,
\begin{equation}\label{eq:gapfact}
\text{if } a>a' \text{ are distinct terms with } a\ge e^{-W}, \text{ then }
a-a'\ge(1-q)\,a\ge(1-q)e^{-W}.
\end{equation}
(If $a=b_i$, $a'=b_{i'}$ with $i<i'$, then $a'\le q^{\,i'-i}a\le qa$.) Two
consequences: an interval of length $2e^{-W}$ contains at most
$\frac{2}{1-q}+1=O_q(1)$ terms that are $\ge e^{-W}$; and a multiplicative window
$[\rho,K\rho]$ contains at most $\frac{\log K}{\log(1/q)}+1=O_{q,K}(1)$ terms.

\begin{proof}[Proof of Lemma~\ref{lem:energy}]
The involution
$(i,i',j,j')\mapsto(i,i',j',j)$ maps the defining condition of $E^-(W)$ to
that of $E(W)$, so $E^-(W)=E(W)$. It remains to prove the asserted bound for
$E(W)$. We classify the quadruples according to whether $i=i'$, whether $j=j'$,
and, in the remaining cases, the signs of $b_i-b_{i'}$ and $d_j-d_{j'}$.

\textbf{(A) Diagonal $i=i'$, $j=j'$.} These number $I(W)J(W)$.

\textbf{(B) $i=i'$, $j\neq j'$.} The condition becomes $|d_j-d_{j'}|\le e^{-W}$.
With $j<j'$, \eqref{eq:gapfact} gives $|d_j-d_{j'}|\ge(1-q)d_j$, so
$d_j\le e^{-W}/(1-q)$; combined with $d_j\ge e^{-W}$ this confines $d_j$ to a
multiplicative window of width $(1-q)^{-1}$, i.e.\ $O_q(1)$ choices of $j$. For each such $j$, the partner $d_{j'}\in[d_j-e^{-W},d_j)$ has only
$O_q(1)$ possible values. Including the reverse ordering $j'>j$ changes only the
constant. Thus case \textbf{(B)} contributes $O_q(I(W))$. Case \textbf{(C)}, in
which $i\neq i'$ and $j=j'$, is symmetric and contributes $O_q(J(W))$.

\textbf{(D) Same sign, $i\neq i'$ and $j\neq j'$.} Say $i<i'$ and $j<j'$ (the case
$i>i'$, $j>j'$ is identical up to an overall sign). Then both $b_i-b_{i'}$ and
$d_j-d_{j'}$ are positive, so by \eqref{eq:gapfact},
\[
\big|(b_i-b_{i'})+(d_j-d_{j'})\big|\ \ge\ (1-q)b_i+(1-q)d_j\ \ge\ 2(1-q)e^{-W}\ >\
e^{-W},
\]
the last inequality because $q<\tfrac12$ gives $2(1-q)>1$. Hence there are no such
quadruples. This is the only place where the strict bound $q<\tfrac12$ is used.

\textbf{(E) Opposite sign, $i\neq i'$ and $j\neq j'$.} Say $i<i'$ and $j>j'$ (the
mirror case $i>i'$, $j<j'$ is identical). Write
$\Delta b:=b_i-b_{i'}\in[(1-q)b_i,\,b_i)$ and
$\Delta d:=d_{j'}-d_j\in[(1-q)d_{j'},\,d_{j'})$, both positive; the condition is
$|\Delta b-\Delta d|\le e^{-W}$. We count by choosing $i$, then $j'$, then $j$,
then $i'$.

\emph{Bulk $i$} (those with $b_i\ge\frac{2}{1-q}e^{-W}$, so that
$e^{-W}\le\frac{1-q}{2}b_i$). From $|\Delta b-\Delta d|\le e^{-W}$ and the ranges
of $\Delta b,\Delta d$,
\[
d_{j'}>\Delta d\ge\Delta b-e^{-W}\ge(1-q)b_i-\tfrac{1-q}{2}b_i=\tfrac{1-q}{2}b_i,
\qquad
(1-q)d_{j'}\le\Delta d\le\Delta b+e^{-W}<2b_i ,
\]
so $d_{j'}\in\big[\tfrac{1-q}{2}b_i,\ \tfrac{2}{1-q}b_i\big]$, a multiplicative
window of width $\frac{4}{(1-q)^2}$: at most $N_0=O_q(1)$ choices of $j'$ once $i$
is fixed. Then $j$ is free ($\le J(W)$ choices). Finally, with $i,j',j$ fixed, the
condition $|\Delta b-\Delta d|\le e^{-W}$ forces
\[
b_{i'}\in\big[\,b_i-(d_{j'}-d_j)-e^{-W},\ b_i-(d_{j'}-d_j)+e^{-W}\,\big].
\]
This interval has length $2e^{-W}$. Since $i'\le I(W)$, every admissible term
satisfies $b_{i'}\ge e^{-W}$. By \eqref{eq:gapfact}, any two distinct terms of the
lacunary sequence lying above $e^{-W}$ are separated by at least $(1-q)e^{-W}$.
Consequently this interval contains at most
\[
\frac{2}{1-q}+1
\]
possible values of $b_{i'}$. Hence, for each fixed choice of $i,j',j$, there are
only $O_q(1)$ admissible choices of $i'$. It follows that the bulk contribution is
at most
\[
I(W)\cdot N_0\cdot J(W)\cdot O_q(1)=O_q\big(I(W)J(W)\big).
\]

\emph{Boundary $i$} (those with $b_i<\frac{2}{1-q}e^{-W}$). There are $O_q(1)$
such $i$ by the window bound. Fix one. Then $\Delta b\le b_i=O_q(e^{-W})$, so
$\Delta d\le\Delta b+e^{-W}=O_q(e^{-W})$; since $\Delta d\ge(1-q)d_{j'}$ this
forces $d_{j'}=O_q(e^{-W})$, hence $j'$ ranges over $O_q(1)$ near-floor indices.
Now $j$ is free ($\le J(W)$). Once $i,j',j$ are fixed, the same interval argument
as in the bulk case shows that $b_{i'}$ lies in an interval of length $2e^{-W}$;
since every admissible $b_{i'}$ is $\ge e^{-W}$, \eqref{eq:gapfact} again gives at
most $O_q(1)$ admissible choices of $i'$. The boundary thus contributes
$O_q(1)\cdot O_q(1)\cdot J(W)\cdot O_q(1)=O_q(J(W))$.

Summing (A)--(E):
\[
E(W)\ \le\ I(W)J(W)+O_q\big(I(W)+J(W)\big)+0+O_q\big(I(W)J(W)\big)
=O_q\big(I(W)J(W)\big),
\]
which is the claim.
\end{proof}

\subsection{The grid principle}\label{sec:gridprinciple}

We now combine the energy estimate with the block principle. The following
proposition is the engine behind both main theorems; the parameter
$\sigma\in\{+1,-1\}$ allows sum and difference grids to be treated
simultaneously.

\begin{proposition}[Grid principle]\label{prop:grid}
Let $(u_i)_{i\ge1}$ and $(v_k)_{k\ge1}$ be lacunary sequences in $(0,\infty)$
with ratios $\le q<\tfrac12$ and counting functions $I$, $J$ satisfying
\eqref{eq:star}. Then for each $\sigma\in\{+1,-1\}$ the grid
\[
G_\sigma:=\{\,u_i+\sigma v_k:\ i,k\ge1\,\}
\]
is not measure universal.
\end{proposition}

\begin{proof}
All points of $G_\sigma$ lie in the bounded interval
$Q:=(-v_1,\ u_1+v_1]$. Fix $W>0$ and consider the multiset of grid points at
scale $\ge e^{-W}$,
\[
P_\sigma(W):=\{\,u_i+\sigma v_k:\ i\le I(W),\ k\le J(W)\,\}\subset Q,
\qquad |P_\sigma(W)|=I(W)J(W),
\]
counted with multiplicity. Partition $Q$ into half-open intervals (``cells'') of
length $e^{-W}$, and let $D(W)$ be the number of cells meeting $P_\sigma(W)$. If
$n_c$ denotes the number of elements of $P_\sigma(W)$ in cell $c$, then
$\sum_c n_c=I(W)J(W)$ and, since two elements of the same cell differ by less
than $e^{-W}$,
\[
\sum_c n_c^2
=\#\big\{(p,p')\in P_\sigma(W)^2:\ p,p'\ \text{in the same cell}\big\}
\ \le\
\begin{cases}
E(W), & \sigma=+1,\\
E^{-}(W), & \sigma=-1,
\end{cases}
\]
because the difference of the pair $(p,p')=\big((i,k),(i',k')\big)$ is
$(u_i-u_{i'})+\sigma(v_k-v_{k'})$. In both cases the right-hand side is at most
$C\,I(W)J(W)$, by Lemma~\ref{lem:energy}. By Cauchy--Schwarz,
\begin{equation}\label{eq:cells}
D(W)\ \ge\ \frac{\big(\sum_c n_c\big)^2}{\sum_c n_c^2}\ \ge\
\frac{\big(I(W)J(W)\big)^2}{C\,I(W)J(W)}\ =\ \frac{I(W)J(W)}{C}.
\end{equation}
Now select one element of $P_\sigma(W)$ from every second occupied cell (i.e.\
from the 1st, 3rd, 5th, \dots occupied cell, ordered left to right). The indices of any two consecutive selected cells differ by at least two, so
points chosen from them are separated by at least one full cell-width. Since each
cell has length $e^{-W}$, the selected points form a set
\[
\Phi_W\subset G_\sigma\cap Q,\qquad
|\Phi_W|\ \ge\ \tfrac12 D(W)\ \ge\ \frac{I(W)J(W)}{2C},\qquad
\min_{x\neq y\in\Phi_W}|x-y|\ \ge\ e^{-W}.
\]
We verify the hypothesis of Proposition~\ref{prop:block}. By \eqref{eq:star}
there are scales $1\le W_1<W_2<\cdots\to\infty$ with
\[
I(W_k)\,J(W_k)\ \ge\ k\,W_k\qquad\text{for all }k\ge1 .
\]
Set $n_k:=\big\lfloor I(W_k)J(W_k)/(2C)\big\rfloor$ and consider the blocks
$\Phi_{W_k}$. Then $n_k\ge kW_k/(2C)-1\to\infty$, $|\Phi_{W_k}|\ge n_k$, and the
gaps in $\Phi_{W_k}$ are at least $e^{-W_k}$. Moreover, for $k\ge 4C$ we have
$I(W_k)J(W_k)\ge kW_k\ge 4C$, hence $n_k\ge I(W_k)J(W_k)/(4C)$ and
\[
\frac{W_k}{n_k}\ \le\ \frac{4C\,W_k}{I(W_k)\,J(W_k)}\ \le\ \frac{4C}{k}\
\longrightarrow\ 0 ,
\]
so $\psi_k:=W_k=o(n_k)$. Proposition~\ref{prop:block} now implies that
$G_\sigma$ is not measure universal.
\end{proof}

\subsection{Proofs of the main theorems}\label{sec:mainproofs}

\begin{proof}[Proof of Theorem~\ref{thm:product}]
By the reduction of Section~\ref{sec:counting} (Lemma~\ref{lem:thinning}) we may
replace $(b_i)$ and $(d_j)$ by thinned subsequences whose ratios are
$\le q<\tfrac12$; condition \eqref{eq:star} is preserved. By
Proposition~\ref{prop:grid},
\[
G_+=\{b_i+d_j\}\quad\text{and}\quad G_-=\{b_i-d_j\}
\]
are not measure universal. Since $G_+\subseteq S_1+S_2$ and
$G_-\subseteq S_1-S_2$, Lemma~\ref{lem:mono} concludes the proof.
\end{proof}

\begin{proof}[Proof of Theorem~\ref{thm:arbitrary}]
Since $S-A=S+(-A)$ and $-A$ is infinite, it suffices to treat $S+A$.

By hypothesis there are $C>0$ and $i_0$ with $-\log b_i\le Ci$ for $i\ge i_0$,
and $(b_i)$ is lacunary with some ratio $q_1\in(0,1)$. Fix $\ell$ with
$q_1^{\ell}<\tfrac12$ and pass to the thinned subsequence $(b_{\ell i})$, which
is lacunary with ratio $q:=q_1^{\ell}<\tfrac12$ and satisfies
$-\log b_{\ell i}\le C\ell i$ for $\ell i \ge i_0$; by
Lemma~\ref{lem:counting}(ii), its counting function --- which we again denote by
$I$ --- satisfies
\begin{equation}\label{eq:Ilinear}
I(W)\ \ge\ cW\qquad\text{for all sufficiently large }W,
\end{equation}
with $c:=1/(2C\ell)$. Since $\{b_{\ell i}\}+A\subseteq S+A$, by
Lemma~\ref{lem:mono} it suffices to show that $\{b_{\ell i}\}+A$ is not measure
universal.

\emph{Case 1: $A$ is unbounded.} Then $\{b_{\ell i}\}+A$ is unbounded,
and no unbounded set is measure universal.

\emph{Case 2: $A$ is bounded.} Being infinite and bounded, $A$ has a finite
accumulation point $t$. At least one of the sets
$A\cap(t,\infty)$, $A\cap(-\infty,t)$ has $t$ as an accumulation point; hence
there exist $\sigma\in\{+1,-1\}$ and a strictly decreasing sequence
$\varepsilon_1>\varepsilon_2>\cdots>0$ with $\varepsilon_j\to0$ and
\[
t+\sigma\varepsilon_j\in A\qquad\text{for all }j .
\]
Extract a lacunary subsequence greedily: let $k_1:=1$ and, given $k_m$, let
$k_{m+1}$ be the least index $k$ with
$\varepsilon_k\le q\,\varepsilon_{k_m}$ (it exists since $\varepsilon_j\to0$).
Then $(v_m):=(\varepsilon_{k_m})$ is lacunary with ratio $\le q<\tfrac12$, and
since it is an infinite null sequence of positive numbers, its counting function
satisfies
\begin{equation}\label{eq:Jinfty}
J(W)\ \ge\ m\qquad\text{whenever}\quad W\ \ge\ -\log v_m,
\qquad\text{so}\qquad J(W)\to\infty.
\end{equation}
Consider the grid $G_\sigma:=\{\,b_{\ell i}+\sigma v_m:\ i,m\ge1\,\}$. Since
$b_{\ell i}+\sigma v_m=b_{\ell i}+(t+\sigma\varepsilon_{k_m})-t$ and
$t+\sigma\varepsilon_{k_m}\in A$, we have
\[
G_\sigma\ \subseteq\ \big(\{b_{\ell i}\}+A\big)-t .
\]
By \eqref{eq:Ilinear} and \eqref{eq:Jinfty},
\[
\frac{I(W)J(W)}{W}\ \ge\ c\,J(W)\ \longrightarrow\ \infty ,
\]
so condition \eqref{eq:star} holds for the pair $(b_{\ell i})$, $(v_m)$, and
Proposition~\ref{prop:grid} shows that $G_\sigma$ is not measure universal. By
Lemma~\ref{lem:mono}, $(\{b_{\ell i}\}+A)-t$ is not measure universal, and by
Lemma~\ref{lem:affine} neither is $\{b_{\ell i}\}+A$.

Finally, for $S=\{2^{-n}\}$ take $b_i=2^{-i}$, so that
$-\log b_i=i\log2=O(i)$.
\end{proof}

\subsection{Proofs of the corollaries}\label{sec:corollaries}

\begin{proof}[Proof of Corollary~\ref{cor:tradeoff}]
(i) By Lemma~\ref{lem:counting}(iii), $I(W)/W^{1/\alpha}\to\infty$ and
$J(W)/W^{1/\beta}\to\infty$. Hence
\[
\frac{I(W)J(W)}{W}
=\frac{I(W)}{W^{1/\alpha}}\cdot\frac{J(W)}{W^{1/\beta}}\cdot
W^{\frac1\alpha+\frac1\beta-1}\ \longrightarrow\ \infty ,
\]
since the first two factors tend to infinity and the exponent of $W$ is
nonnegative. This is \eqref{eq:star}, and Theorem~\ref{thm:product} applies.

(ii) By Lemma~\ref{lem:counting}(ii), $I(W)\ge c\,W^{1/\alpha}$ and
$J(W)\ge c'\,W^{1/\beta}$ for all large $W$, so
$I(W)J(W)/W\ge cc'\,W^{\frac1\alpha+\frac1\beta-1}\to\infty$, the exponent now
being strictly positive. Theorem~\ref{thm:product} applies again.
\end{proof}

\begin{proof}[Proof of Corollary~\ref{cor:symmetric}]
The first statement is the case $\alpha=\beta=2$ of
Corollary~\ref{cor:tradeoff}(i). For the specializations, apply it with
$S_1=S_2=S$ and $(b_i)=(d_j)$ the given subsequence. For a geometric sequence
$S=\{ar^n\}$ ($a\neq0$, $0<|r|<1$): by Lemma~\ref{lem:affine} (with $\rho=-1$,
$\tau=0$) we may assume $a>0$; take $b_k=ar^k$ if $r\in(0,1)$, or the even powers
$b_k=ar^{2k}$ if $r\in(-1,0)$; in either case $(b_k)\subset S$ is lacunary with
$-\log b_k=O(k)=o(k^2)$. For $\{2^{-n^{\alpha}}\}$ with $1\le\alpha<2$, take
$b_k=2^{-k^{\alpha}}$, which is lacunary and satisfies
$-\log b_k=k^{\alpha}\log2=o(k^2)$. If $0<\alpha<1$, then
$2^{-(n+1)^{\alpha}}/2^{-n^{\alpha}}\to1$, so
Lemma~\ref{lem:extraction} supplies a lacunary subsequence with
$-\log b_k=O(k)=o(k^2)$. This proves \eqref{eq:stretched}.
\end{proof}

For Corollary~\ref{cor:concrete}(b) we use the standard extraction of a usable
subsequence.

\begin{lemma}\label{lem:extraction}
Let $(a_n)$ be decreasing with $a_n\to0$ and $a_{n+1}/a_n\ge\eta>0$ for all
large $n$. Then for every $q\in(0,1/2)$ there is a subsequence
$(b_k)=(a_{n_k})$ with $\eta q\,b_k<b_{k+1}\le q\,b_k$ for all $k$. In
particular $(\eta q)^{k-1}b_1<b_k\le q^{k-1}b_1$, so $-\log b_k=O(k)$.
\end{lemma}

\begin{proof}
Choose $n_1$ beyond the range where the ratio bound fails. Having chosen $n_k$,
let $n_{k+1}$ be the least index with $a_{n_{k+1}}\le q\,a_{n_k}$ (it exists
since $a_n\to0$), and set $b_k=a_{n_k}$. Then $b_{k+1}\le q\,b_k$, while by
minimality $a_{n_{k+1}-1}>q\,b_k$, so
$b_{k+1}=a_{n_{k+1}}\ge\eta\,a_{n_{k+1}-1}>\eta q\,b_k$.
\end{proof}

\begin{proof}[Proof of Corollary~\ref{cor:concrete}]
(a) Suppose first $a>0$. If $r\in(0,1)$, then $(b_i)=(ar^i)$ is lacunary with
ratio $r$ and $-\log b_i=i\log(1/r)-\log a=O(i)$; if $r\in(-1,0)$, the even
powers $(ar^{2i})$ are lacunary with ratio $r^2$ and the same decay control. In
either case Theorem~\ref{thm:arbitrary} applies and gives the claim for
$\{ar^n\}\pm A$. If $a<0$, apply the case just treated to
$\{(-a)r^n\}\pm(-A)=-\big(\{ar^n\}\pm A\big)$ and use
Lemma~\ref{lem:affine} with $\rho=-1$.

(b) The hypothesis provides $\eta>0$ with $a_{n+1}/a_n\ge\eta$ eventually.
Lemma~\ref{lem:extraction} (with any $q<\tfrac12$) yields a lacunary subsequence
$(b_k)\subset\{a_n\}$ with $-\log b_k=O(k)$, and Theorem~\ref{thm:arbitrary}
applies. Taking $A=\{a_n\}$ recovers the double sumset.
\end{proof}

\begin{proof}[Proof of Corollary~\ref{cor:stretched}]
If $\alpha_1\le1$, then
\[
\liminf_{n\to\infty}
\frac{2^{-(n+1)^{\alpha_1}}}{2^{-n^{\alpha_1}}}>0,
\]
so Corollary~\ref{cor:concrete}(b) applies with
$A=\{2^{-n^{\alpha_2}}\}$, or with any infinite set. The case
$\alpha_2\le1$ is symmetric. If $\alpha_1,\alpha_2>1$ with
$1/\alpha_1+1/\alpha_2>1$, choose $\delta>0$ small enough that
\[
\frac{1}{\alpha_1+\delta}+\frac{1}{\alpha_2+\delta}\ \ge\ 1 ,
\]
which is possible by continuity. Since
$n^{\alpha_m}\log2=o\big(n^{\alpha_m+\delta}\big)$ for $m=1,2$,
Corollary~\ref{cor:tradeoff}(i) applies with $\alpha=\alpha_1+\delta$,
$\beta=\alpha_2+\delta$.
\end{proof}

\subsection{A packing variant: lacunarity on one side only}\label{sec:packing}

In Lemma~\ref{lem:energy}, the lacunarity of the first sequence is used in two
ways: to exclude the same-sign quadruples (case (D)) and to localize the index
$i'$ (cases (C) and (E)). Both uses survive if the first sequence is replaced by
an \emph{arbitrary} well-separated finite set, provided one localizes $d_{j'}$
around the difference $\Delta u$ rather than around $u_i$. This observation
yields Theorem~\ref{thm:packing}. Note that the strict bound $q<\tfrac12$ is not
needed here: the same-sign case is now excluded by the separation of the first
factor alone.

\begin{lemma}\label{lem:packenergy}
Let $W>0$, let $U=\{u_1,\dots,u_N\}\subset\R$ be a finite set with
$|u_a-u_{a'}|\ge 2e^{-W}$ for all $a\neq a'$, and let $(d_j)$ be lacunary with
ratio $q\in(0,1)$ and counting function $J$; write $J=J(W)$. For
$\sigma\in\{+1,-1\}$ set
\[
E_\sigma(W):=\#\Big\{(a,a',j,j')\in\{1,\dots,N\}^2\times\{1,\dots,J\}^2:\
\big|(u_a-u_{a'})+\sigma(d_j-d_{j'})\big|\le e^{-W}\Big\}.
\]
Then there is $C=C(q)$ with
\[
E_\sigma(W)\ \le\ NJ+C\,N^2+C\,N .
\]
In particular, if $N\le J$ then $E_\sigma(W)\le (1+2C)\,NJ$.
\end{lemma}

\begin{proof}
By the involution $(a,a',j,j')\mapsto(a,a',j',j)$, as in
Lemma~\ref{lem:energy}, it suffices to treat $\sigma=+1$. We again classify by
the signs of $u_a-u_{a'}$ and $d_j-d_{j'}$.

(A) $a=a'$, $j=j'$: exactly $NJ$ quadruples.

(B) $a=a'$, $j\neq j'$: the condition is $|d_j-d_{j'}|\le e^{-W}$, and exactly as
in case (B) of Lemma~\ref{lem:energy} (which used only the lacunarity of
$(d_j)$, valid for any $q\in(0,1)$) there are $O_q(1)$ admissible pairs
$(j,j')$; total $O_q(N)$.

(C) $a\neq a'$, $j=j'$: the condition is $|u_a-u_{a'}|\le e^{-W}$, impossible by
the $2e^{-W}$-separation of $U$.

(D) $a\neq a'$, $j\neq j'$, same signs: then
$|(u_a-u_{a'})+(d_j-d_{j'})|=|u_a-u_{a'}|+|d_j-d_{j'}|\ge 2e^{-W}>e^{-W}$, so
there are no such quadruples.

(E) $a\neq a'$, $j\neq j'$, opposite signs: relabelling, we may assume
$\Delta u:=u_a-u_{a'}>0$ and $\Delta d:=d_{j'}-d_j>0$ with $j'<j$, the condition
being $|\Delta u-\Delta d|\le e^{-W}$; the mirror case contributes an equal
count. Choose the ordered pair $(a,a')$: at most $N^2$ ways, and $\Delta u\ge
2e^{-W}$ is then fixed. The condition forces
$\Delta d\in[\Delta u-e^{-W},\,\Delta u+e^{-W}]\subset[\tfrac12\Delta u,\,
\tfrac32\Delta u]$. Since $\Delta d=d_{j'}-d_j\in[(1-q)d_{j'},\,d_{j'})$, we get
\[
d_{j'}\in\Big(\Delta d,\ \frac{\Delta d}{1-q}\Big]\subset
\Big(\tfrac12\Delta u,\ \frac{3\Delta u}{2(1-q)}\Big],
\]
a multiplicative window of bounded width: at most $O_q(1)$ choices of $j'$.
Given $(a,a',j')$, the value $d_j=d_{j'}-\Delta d$ lies in an interval of length
$2e^{-W}$; since all admissible $d_j$ are $\ge e^{-W}$ and, by
\eqref{eq:gapfact}, pairwise $(1-q)e^{-W}$-separated, there are at most $O_q(1)$
choices of $j$. Total: $O_q(N^2)$.

Summing (A)--(E) gives the claim; if $N\le J$ then $N^2\le NJ$ and $N\le NJ$.
\end{proof}

\begin{proof}[Proof of Theorem~\ref{thm:packing}]
Since $S_1-S_2\supseteq\{u-d_j:\ u\in S_1\cap Q_0,\ j\ge1\}$ and similarly for
$S_1+S_2$, by Lemma~\ref{lem:mono} it suffices to prove that for each
$\sigma\in\{+1,-1\}$ the set
\[
G_\sigma:=\{\,u+\sigma d_j:\ u\in S_1\cap Q_0,\ j\ge1\,\}
\]
is not measure universal. All points of $G_\sigma$ lie in the bounded interval
$Q:=[\inf Q_0-d_1,\ \sup Q_0+d_1]$.

Fix $W>0$ and set $\widetilde N(W):=\min\big(N(W),J(W)\big)$. Choose a
$2e^{-W}$-separated set $U_W\subseteq S_1\cap Q_0$ with
$|U_W|=\widetilde N(W)$ --- the choice may depend on $W$; the blocks produced at
different scales need not be related --- and consider the multiset
\[
P_\sigma(W):=\{\,u+\sigma d_j:\ u\in U_W,\ j\le J(W)\,\}\subset Q,\qquad
|P_\sigma(W)|=\widetilde N(W)\,J(W).
\]
Cover $Q$ by the consecutive half-open cells
$[\inf Q+(m-1)e^{-W},\ \inf Q+me^{-W})$, $m\ge1$, so that every point of $Q$
lies in exactly one cell and all cells have length $e^{-W}$. Two elements of the
same cell differ by less than $e^{-W}$, so the number of same-cell ordered pairs
is at most $E_\sigma(W)\le(1+2C)\widetilde N(W)J(W)$ by
Lemma~\ref{lem:packenergy} (applicable since $\widetilde N(W)\le J(W)$). By
Cauchy--Schwarz, as in \eqref{eq:cells}, the number of occupied cells is at
least $\widetilde N(W)J(W)/(1+2C)$, and selecting one point from every second
occupied cell produces
\[
\Phi_W\subset G_\sigma\cap Q,\qquad
|\Phi_W|\ \ge\ \frac{\widetilde N(W)\,J(W)}{2(1+2C)},\qquad
\min_{x\neq y\in\Phi_W}|x-y|\ \ge\ e^{-W}.
\]
Finally, by \eqref{eq:packstar} there are scales $1\le W_1<W_2<\cdots\to\infty$
with $\widetilde N(W_k)J(W_k)\ge kW_k$, and exactly as at the end of the proof
of Proposition~\ref{prop:grid} the blocks $\Phi_{W_k}$ satisfy the hypothesis of
Proposition~\ref{prop:block}. Hence $G_\sigma$ is not measure universal.
\end{proof}

\begin{proof}[Proof of Theorem~\ref{thm:arbitrary} via Theorem~\ref{thm:packing}]
By Lemma~\ref{lem:counting}(ii), the lacunary sequence in $S$ has
$J(W)\ge cW$ for large $W$. If $A$ is unbounded, then $S+A$ and $S-A$ are
unbounded and we are done; otherwise $A$ has infinitely many points in some
bounded interval $Q_0$, and its packing function there satisfies
$N(W)\to\infty$. Then
\[
\frac{\min\big(N(W),J(W)\big)\,J(W)}{W}\ \ge\ c\,\min\big(N(W),J(W)\big)\
\longrightarrow\ \infty ,
\]
and Theorem~\ref{thm:packing} with $S_1=A$, $S_2=S$ shows that $A+S=S+A$ and
$A-S$ are not measure universal; since $S-A=-(A-S)$, Lemma~\ref{lem:affine}
(with $\rho=-1$) gives the claim for $S-A$ as well.
Note that no extraction of a one-sided monotone sequence in $A$ is needed: the
separated sets $U_W$ are taken directly inside $A\cap Q_0$.
\end{proof}

\begin{remark}[Why the cap at $J(W)$ is natural]\label{rem:capNJ}
The term $C N^2$ in Lemma~\ref{lem:packenergy} reflects the fact that an
arbitrary separated set $U$ may have many repeated differences aligned with
differences of $(d_j)$. For example, if $U$ contains a long arithmetic
progression whose step equals one of the differences $d_{j'}-d_j$, then a
single choice of $(j,j')$ already produces many quadruples in case (E).
Without an additional hypothesis controlling the additive structure of $U$, an
$N^2$ bound is therefore the natural uniform estimate.

This loss is harmless in the range where Theorem~\ref{thm:packing} carries new
information. If
\[
\limsup_{W\to\infty}\frac{N(W)}{W}=\infty,
\]
then $S_1$ alone is not measure universal: apply
Proposition~\ref{prop:block} directly to the separated sets $U_{W_k}$, with
$n_k=N(W_k)$ and $W_k=o(n_k)$. Thus the relevant regime is $N(W)=O(W)$. With
an energy bound of order $NJ+N^2$, the Cauchy--Schwarz step
yields a separated set of order
\[
\frac{(NJ)^2}{NJ+N^2}\asymp \min(N,J)J,
\]
which explains the cap in \eqref{eq:packstar}.
\end{remark}

\begin{remark}[Lacunarity on one factor only]\label{rem:clustered}
Theorem~\ref{thm:packing} is not subsumed by Theorem~\ref{thm:product}: the first
factor enters only through its packing function $N$, which may far exceed the
counting function of any lacunary subsequence it contains. Concretely, let
\[
S_1:=\bigcup_{k\ge1}\Big\{2^{-2^k}+m\cdot 4^{-2^k}:\ 0\le m\le 2^k\Big\},
\qquad
S_2:=\{2^{-j^{3/2}}\}_{j\ge1}.
\]
The $k$-th cluster of $S_1$ contains $2^k+1$ points spaced
$4^{-2^k}$ apart in an interval of length
$2^k4^{-2^k}\ll 2^{-2^k}$. At the scales
$W_k=(2^{k+1}+1)\log 2$, the spacing equals $2e^{-W_k}$, and hence
$N(W_k)\ge 2^k+1\gtrsim W_k$. On the other hand, any lacunary subsequence of
$S_1$ uses at most
$O(1)$ points per cluster: the diameter of the $k$-th cluster is
$o(2^{-2^k})$, so for any fixed lacunarity ratio it contains at most one selected
point once $k$ is large. Such a subsequence therefore has counting function
$O(\log W)$. With $S_2$
lacunary and $J(W)\asymp W^{2/3}$, condition \eqref{eq:star} fails along this pair
($I(W)J(W)/W\to0$), yet
\[
\frac{\min(N(W),J(W))\,J(W)}{W}\ \asymp\ \frac{W^{2/3}\cdot W^{2/3}}{W}
= W^{1/3}\ \longrightarrow\ \infty ,
\]
so Theorem~\ref{thm:packing} applies where Theorem~\ref{thm:product} does not.
The point is methodological: the near-energy estimate, in the packing form of
Lemma~\ref{lem:packenergy}, reaches a clustered first factor --- one with no
usable lacunary subsequence --- through metric mass alone, with no
relative-density input. The conclusion itself is \emph{not} new for this
particular pair: the clusters of $S_1$ are arithmetic progressions whose longest
gap is a vanishing fraction of the cluster. Thus
$\inf_{u<v}E(u,v)/(v-u)=0$, where $E(u,v)$ denotes the length of the longest
component of $(u,v)\setminus S_1$. The relative-density criterion of Humke and
Laczkovich~\cite{HL} (recorded in \cite[Theorem~2.25]{Svetic}) therefore already
shows that $S_1$ is not measure universal, and hence neither is $S_1\pm S_2$.
What is
new is that the additive-energy method of this paper reaches such pairs without
invoking any relative-density argument.
\end{remark}

\section{Further remarks}

\begin{remark}\label{rem:fullmeasure}
All results of this paper consume Theorem~\ref{thm:kol} as a black box: the
proofs construct the required blocks and use nothing else about the criterion.
Consequently, any strengthening of Theorem~\ref{thm:kol} transfers verbatim to
Theorems~\ref{thm:product}--\ref{thm:packing} and their corollaries. Two caveats
delimit what such strengthenings can say. First, for the sumsets of sequences
considered here --- which are \emph{countable} --- an avoiding set of \emph{full}
measure is impossible: as observed in \cite[\S1.3]{KolCantor}, if
$E\subseteq[0,1]$ has measure $1$ and $A$ is bounded and countable, then almost
every translate of a suitably contracted copy of $A$ lies in $E$. (This is in
contrast with uncountable sets, for which full-measure avoidance can hold
\cite[Theorem~1.3]{KolCantor}, \cite{GLW,SY}.) Natural quantitative substitutes are avoiding sets of measure arbitrarily
close to $1$, which the construction in \cite{Kol97} already provides, and
statements in which the exceptional set of dilations is null, as in
\cite[Theorem~2]{Kol97}. Second, the
``large sets'' of \cite{KP} --- sets of density $\ge 1-\epsilon$ in every unit
interval avoiding copies of a given \emph{unbounded} sequence --- concern the
Erd\H{o}s problem in the large, a different regime from the null sequences
studied here; see also \cite{GMY}.
\end{remark}

\section*{Acknowledgements}

A.~I. was supported in part by the National Science Foundation under NSF
DMS-2154232. A.~Y. was supported in part by the Natural Sciences and
Engineering Research Council of Canada, NSERC (GR030571 and GR030540).

\end{document}